\documentclass[a4paper,12pt,reqno]{amsart}
\usepackage{amsthm, mathrsfs,amssymb,amsmath,}
\usepackage{enumerate}
\usepackage{mathtools}
\usepackage{orcidlink}
\makeatletter
\@namedef{subjclassname@2010}{%
	\textup{2010} Mathematics Subject Classification}
\makeatother

\usepackage{hyperref}
\allowdisplaybreaks
\numberwithin{equation}{section}

\usepackage{graphicx}
\usepackage{enumitem}

\renewcommand{\arraystretch}{2.5}

\newtheorem{theorem}{Theorem}[section]
\newtheorem{lemma}[theorem]{Lemma}
\usepackage{xcolor}

\theoremstyle{definition}
\newtheorem{definition}[theorem]{Definition}

\newtheorem{proposition}[theorem]{Proposition}

\theoremstyle{remark}

\numberwithin{equation}{section}
\newcommand{\F}{\mathcal{F}}

\newcommand{\N}{\mathbb{N}}
\newcommand{\G}{\mathcal{G}}

\newcommand{\HH}{\mathcal{H}}

\def\intl{\int\limits}

\def\suml{\sum\limits}

\begin{document}
	
	\title[The Dimension Spectrum of Continued fractions]{Dimension Spectrum of Continued fraction Expansions with Coefficients restricted to the Fibonacci Sequence}
	
	
	
	\author{Arpit Dawar\orcidlink{0009-0006-0789-4205}}
	\address{Department of Mathematics, Indian Institute of Technology Delhi, New Delhi, India 110016}
	\email{arpitdawar.m@gmail.com}
	
		\author{Amit Priyadarshi\orcidlink{0000-0002-7706-1587}}
	\address{Department of Mathematics, Indian Institute of Technology Delhi, New Delhi, India 110016}
	
	\email{priyadarshi@maths.iitd.ac.in}
	
	
	%

	
	
	
	\begin{abstract}
    In this paper, we analyze the structure of the dimension spectrum of continued fraction expansions with coefficients restricted to the generalized Fibonacci sequence. Let $F_{(a_1,a_2)}$ denote the generalized Fibonacci sequence starting with the positive integers $a_1<a_2$. We prove that the continued fractions whose digits lie in $F_{(a_1,a_2)}$ have full dimension spectrum for every $(a_1,a_2)$ such that $a_1 \geq2$, or $a_1=1$ and $a_2\geq3$. On the other hand, using the numerical tools developed by Falk and Nussbaum, we show that the dimension spectrum has a gap for continued fractions with digits restricted to each of the sets $F_{(1,2)}$ and $F_{(2,1)}$, where $F_{(2,1)}$ denotes the set of Lucas numbers. Moreover, for $F_{(1,2)}$ and $F_{(2,1)}$, we prove that the dimension spectrum always contains a non-trivial interval.
\end{abstract}

	\maketitle

	
    \section{Introduction}
Let $[a_1,a_2,\ldots]$ denote the continued fraction expansion of an irrational number $x\in[0,1]$. That is, 
$$x=[a_1,a_2,\ldots]:=\frac{1}{a_1+\frac{1}{a_2+\frac{1}{a_3+\cdots}}},$$
    where the coefficients $a_i\in\N$ for all $i$. Let us consider the set of continued fraction expansions whose coefficients lie in a set $\Lambda\subseteq\N$. We denote this set by $J_{\Lambda}$, i.e. for any set $\Lambda\subseteq \N$,
    $$J_{\Lambda}:=\{x=[a_1,a_2,\ldots]:a_i\in\Lambda\}\subset[0,1].$$
    The study of these fractal sets and their dimensional properties has become of significant interest due to their applications to the Markoff and Lagrange spectra \cite{MR647383, MR803348, MR1010419}. The Hausdorff dimension of the set $J_{\Lambda}$, denoted by  $\dim_H(J_{\Lambda})$, has received considerable attention in the literature, see, for example, \cite{MR647383,MR803348,MR4461209,MR4878,MR1034198, MR1154044,MR3742587,MR1487636,MR2846362}. For the definitions and basic properties of Hausdorff dimension, see \cite{MR3236784}. \par
    In this paper, we investigate the dimension spectrum of an infinite subset $\Lambda\subseteq\N$, which is defined as 
    $$DS(\Lambda)=\{\dim_H(J_A):A\subseteq\Lambda\}.$$ 
  The structure of the dimension spectrum has been investigated for different infinite subsets of $\N$, see \cite{CHITANGA2026,MR3960790,MR4068257,MR4461210,MR2069367,jurga2021dimension,MR2197868}. Initially, the dimension spectrum of $\N$ was studied by Kesseb\"ohmer and Zhu in \cite{MR2197868}, where it was shown that $DS(\N)=[0,1]$. This proved the Texan conjecture, originally posed by Hensley in \cite{MR1387719} and independently by Mauldin and Urbański in \cite{MR1487636}. Subsequently, in \cite{MR4068257}, Chousionis, Leykekhman, and Urbański analyzed the dimension spectrum of various sets, including the set of prime numbers, the set of squares, and sets of consecutive powers of integers $q\geq2$. \par
  Recently, in \cite{CHITANGA2026}, Chitanga, Lemmens, and Nussbaum obtained a more general result  covering many of these sets. More precisely, it was shown that if $F=\{a_1,a_2,\ldots\}\subset\N$ with $2\leq a_1<a_2<\ldots$ and $a_na_m\geq a_{n+m}$ for all $n,m\in\N$, then $$\left[0,\dim_H(J_F)\right]=DS(F).$$
  This motivates us to investigate the dimension spectrum of sets that do not satisfy the condition $a_na_m\geq a_{n+m}$ for all $n,m\in\N$. One such example is the generalized Fibonacci sequence. Let $F_{(a_1,a_2)}$ denote the generalized Fibonacci sequence starting with the positive integers $a_1<a_2$, that is, $$F_{(a_1,a_2)}=\{(a_1,a_2,a_3,\ldots):a_1<a_2,\ a_n=a_{n-1}+a_{n-2}\ \forall\ n\geq3\}.$$
  We analyze the structure of the dimension spectrum of the set $F_{(a_1,a_2)}$ for integers $1\leq a_1<a_2$. In section 3, we prove that  the dimension spectrum of the generalized Fibonacci sequence starting with an integer $\geq2$ is full. More precisely, we establish the following result. 

\begin{theorem}\label{th1.1}
    If $F_{(a_1,a_2)}=\{a_1,a_2,\ldots\}\subset\N$ with $2\leq a_1<a_2$ and  $a_n=a_{n-1}+a_{n-2}$ for all $n\geq3$, then $$\left[0,\dim_{H}(J_{F_{(a_1,a_2)}})\right]=DS(F_{(a_1,a_2)}).$$
\end{theorem}

In section 4, we investigate the dimension spectrum of the generalized Fibonacci sequence $F_{(a_1,a_2)}$ in the case $a_1=1$. We show that the dimension spectrum is full whenever $a_2\geq3$. In contrast, for $F_{(1,2)}$, the dimension spectrum exhibits a gap. Moreover, we show that the dimension spectrum of the set of Lucas numbers has a similar structure to that of $F_{(1,2)}$. More precisely, we prove the following results.
\begin{theorem}\label{th1.2}
    Let $F_{(a_1,a_2)}=\{a_1,a_2,\ldots\}\subset\N$ with $a_n=a_{n-1}+a_{n-2}$ for all $n\geq3$. Then the following statements hold.
    \begin{enumerate}
    \item[(i)] If $a_1=1$ and $a_2\geq3$, then $$\left[0,\dim_{H}(J_{F_{(a_1,a_2)}})\right]=DS(F_{(a_1,a_2)}).$$
        \item[(ii)] There exists $0<s<\dim_{H}(J_{F_{(1,2)}})$ such that $[0,s]\subset DS(F_{(1,2)})$.
        \item[(iii)] There exist  $0<a<b<\dim_{H}(J_{F_{(1,2)}})$ such that $(a,b)\cap DS(F_{(1,2)})=\varnothing$.
    \end{enumerate}
\end{theorem}

\begin{theorem}\label{th1.3}
    If $F_{(2,1)}=\{2,1,\ldots\}\subset\N$ with $a_n=a_{n-1}+a_{n-2}$ for all $n\geq3$, then
    \begin{enumerate}
        \item[(i)] there exists $0<s<\dim_{H}(J_{F_{(2,1)}})$ such that $[0,s]\subset DS(F_{(2,1)})$,
        \item[(ii)] there exist  $0<a<b<\dim_{H}(J_{F_{(2,1)}})$ such that $(a,b)\cap DS(F_{(2,1)})=\varnothing$.
    \end{enumerate}
\end{theorem}

We recall that the set of continued fraction expansions can be described by an infinite conformal iterated function system (CIFS). Let $\Lambda=\{a_1,a_2,\ldots\}\subseteq\N$ with $a_1<a_2<\cdots$. Let $X$ be the closed interval $\left[0,\frac{1}{a_1}\right]$. For $a_i\in\Lambda$, define $\phi_i(x)=\frac{1}{a_i+x}$. It is easy to verify that $\phi_i(X)\subset X$ for every $i\in\N$. Note that if $a_1=1$, then $\phi_1$ is not a contraction map on $X$. Thus, we consider the system of second level maps, that is $\{\phi_i\circ\phi_j:i,j\in\N\}$. Then the system $S:=\{\phi_i\circ\phi_j:i,j\in\N\}$ satisfies the properties of a CIFS; for a detailed discussion on CIFS and the associated limit set, see \cite{MR1387085,MR1487636,MR2846362}. It can be shown that $J_{\Lambda}$ is precisely the limit set for this CIFS.\par
The Hausdorff dimension of a set of continued fraction expansions can be estimated using Perron–Frobenius operators. Although these operators are defined in a more general setting in the literature, see \cite{MR1806843,MR2846362}, we restrict ourselves to the continued fraction case and define them accordingly. Let $\Lambda=\{a_1,a_2,\ldots\}\subseteq\N$ with $a_1<a_2<\ldots$, and $X=\left[0,\frac{1}{a_1}\right]$. For $s\geq0$, the Perron-Frobenius operator associated with $\Lambda$ is defined as 
$$(L_{s,\Lambda}f)(x)=\suml_{i\geq1}\left(\frac{1}{a_i+x}\right)^{2s}f\left(\frac{1}{a_i+x}\right)$$ for $f\in C(X)$ and $x\in X$, where $C(X)$ denotes the Banach space of continuous real-valued functions defined on $X$ with sup norm. It is known that $L_{s,\Lambda}:C(X)\to C(X)$ is a positive bounded linear operator for all $s>s_0$, where $s_0=\inf\{s\geq0:\suml_{i=1}^{\infty}a_i^{-2s}<\infty\}.$ In the case where $\Lambda\subseteq F_{(a_1,a_2)}$ we have that $s_0=0$. The relation between the Hausdorff dimension of $J_{\Lambda}$ and the spectral radius of $L_{s,\Lambda}$ was investigated in \cite{MR3827801,MR1806843} when $\Lambda$ is finite, and in \cite{MR2846362} when $\Lambda$ is infinite. We recall below the results that will be important for our analysis.
\begin{theorem}\textup{(\cite[Theorem 3.1]{MR3827801})}\label{th1.4}
    Let $\Lambda=\{a_1,a_2,\ldots,a_p\}$ be a finite subset of $\N$ with $a_1<a_2<\ldots<a_p$, and $X=\left[0,\frac{1}{a_1}\right]$. Let $s>0$. Then the following assertions hold.
    \begin{enumerate}
        \item The operator $L_{s,\Lambda}:C(X)\to C(X)$,  defined as above, has a unique, strictly positive eigenvector $v_{s,\Lambda}$ with $L_{s,\Lambda}v_{s,\Lambda}=\lambda_{s,\Lambda}v_{s,\Lambda}$, where  $\lambda_{s,\Lambda}>0$, and $\lambda_{s,\Lambda}$ is the spectral radius of $L_{s,\Lambda}$, denoted as $r(L_{s,\Lambda})$. Moreover, the map $s\to\lambda_{s,\Lambda}$ is strictly decreasing and continuous.
        \item The function $v_{s,\Lambda}$ is decreasing on $X$, and for all $x,y\in X$, $$ v_{s,\Lambda}(x)\leq v_{s,\Lambda}(y)e^{\frac{2s|x-y|}{a_1}}.$$
        \item The unique value $s$ such that $\lambda_{s,\Lambda}=1$ is equal to $\dim_H(J_{\Lambda})$.
    \end{enumerate}
\end{theorem}
In the case when $|\Lambda|=\infty$, the following statements are proved in \cite[Lemma 5.4, Theorem 5.11]{MR2846362}.
\begin{theorem}\label{th1.5}
Let $\Lambda\subseteq\N$ with $|\Lambda|=\infty$, and  let $\sigma_{\infty}=\inf\{s>0:r(L_{s,\Lambda})<1\}$.
    \begin{enumerate}
        \item The map $s\to r(L_{s,\Lambda})$ is continuous and strictly decreasing for $s>s_0$.
        \item $\dim_H(J_{\Lambda})=\sigma_{\infty}$.
    \end{enumerate}
\end{theorem}

\section{Preliminary Results}
In this section, we recall some preliminary results that we will use in our work. First, we prove two observatory lemmas about the Fibonacci sequence that will be used throughout the paper.
\begin{lemma}\label{l2.1}
Let $F_{(a_1,a_2)}$ be the generalized Fibonacci sequence starting with the positive integers $a_1<a_2$. Let $k\geq2$ be an integer. Then $$\frac{a_k}{a_{k+j}}\geq \frac{1}{2^j}\ \forall\ j\in\N.$$
\begin{proof}
   Observe that for any $n\geq3$, $a_n=a_{n-1}+a_{n-2}\leq 2a_{n-1}$. Thus, we have $a_{k+j}\leq 2a_{k+j-1}\leq 2^2a_{k+j-2}\leq\cdots\leq 2^ja_{k}$. 
\end{proof}
\end{lemma}
\begin{lemma}\label{l2.2}
   Let $F_{(a_1,a_2)}$ be the generalized Fibonacci sequence starting with the positive integers $a_1<a_2$. Let $k\geq2$ be an integer. Then, for each $j\in\N$, we have  $$\frac{a_k}{a_{k+j}}\geq \min\left\{\frac{a_2}{a_{2+j}},\frac{a_3}{a_{3+j}}\right\}.$$ 
   \begin{proof}
       We will show that either the even subsequence of $\left\{\frac{a_k}{a_{k+j}}\right\}_{k\geq1}$ is increasing and the odd subsequence of $\left\{\frac{a_k}{a_{k+j}}\right\}_{k\geq1}$ is decreasing, or vice versa. Then the result follows due to the convergence of the sequence $\left\{\frac{a_k}{a_{k+j}}\right\}_{k\geq1}$. We first consider the case  $j=1$. In this case, we have
       \begin{align*}
           \frac{a_{k+2}}{a_{k+3}}-\frac{a_k}{a_{k+1}}&=\frac{a_{k+2}a_{k+1}-a_ka_{k+3}}{a_{k+3}a_{k+1}}\\
           &=\frac{a_{k+1}^2+a_{k}a_{k+1}-a_ka_{k+3}}{a_{k+3}a_{k+1}}\\
           &=\frac{a_{k+1}^2-a_ka_{k+2}}{a_{k+3}a_{k+1}}\\
           &=\frac{(-1)^{k+1}D}{a_{k+3}a_{k+1}}
       \end{align*}
       where $D=a_2^2-a_1a_2-a_1^2$. Thus, in the case where $D>0$, the odd subsequence is increasing while the even subsequence is decreasing. Conversely, when $D<0$, the odd subsequence decreases and the even subsequence increases. Without loss of generality, assume that 
       $$ \frac{a_{k+2}}{a_{k+2+n}}-\frac{a_k}{a_{k+n}}>0$$
       for all $n\leq j$. Now, for $n=j+1$, we have
       \begin{align*}
           \frac{a_{k+2+(j+1)}}{a_{k+2}}&=\frac{a_{k+2+j}}{a_{k+2}}+\frac{a_{k+2+j-1}}{a_{k+2}}\\
           &<\frac{a_{k+j}}{a_{k}}+\frac{a_{k+j-1}}{a_{k}}\\
           &=\frac{a_{k+(j+1)}}{a_{k}}.
       \end{align*}
       Thus, we have $$ \frac{a_{k+2}}{a_{k+2+(j+1)}}-\frac{a_k}{a_{k+(j+1)}}>0.$$
       This completes the proof.
   \end{proof}
\end{lemma}
Now, we state some well-known results related to the Hausdorff dimension of the set of continued fraction expansions with digits restricted to a subset of $\N$. The proofs of the following results can be found in \cite{CHITANGA2026,MR4068257,MR4461210,MR1387085,MR4963813}.
\begin{theorem}\textup{(\cite{MR1387085})}\label{th2.3}
    Let $F$ be an infinite subset of $\N$. If $\{F_n\}_{n\geq1}$ is a nested sequence of finite subsets of $F$ and $\cup_{n}F_n=F$, then $$\lim\limits_{n}\dim_H(J_{F_n})=\dim_H(J_{F}).$$
\end{theorem}
\begin{proposition}\textup{(\cite{MR4068257})}\label{p2.4}
    If $\Lambda_1,\Lambda_2\subset\N$ and there is a non-decreasing bijection $f:\Lambda_1\to\Lambda_2$, then $$\dim_H(J_{\Lambda_2})\leq\dim_H(J_{\Lambda_1}).$$
\end{proposition}
\begin{lemma}\textup{(\cite{CHITANGA2026})}\label{l2.5}
If $\Lambda\subset\N$ is finite with $|\Lambda|\geq2$, and $\sigma=\dim_H(J_{\Lambda})$, then there exists a constant $C_{\Lambda}>1$ such that for all $n\in\N\setminus\Lambda$ we have that $$\sigma+C_{\Lambda}^{-1}n^{-2\sigma}\leq\dim_H(J_{{\Lambda}\cup\{n\}})\leq\sigma+C_{\Lambda}n^{-2\sigma}.$$
\end{lemma}
We now recall some elementary lemmas from \cite{CHITANGA2026} that are useful for estimating the spectral radius of the positive operators $L_{s,\Lambda}$.
\begin{lemma}\textup{(\cite[Lemma 2.1]{CHITANGA2026})}\label{l2.6}
    Let $f,g\in C([a,b])$ be strictly positive. For each $0<\alpha<1$, there exists a $\beta\in(\alpha,1)$ such that $f+\alpha g\leq\beta(f+g)$. Similarly, for each $\alpha>1$, there exists a $\beta\in(1,\alpha]$ such that $\beta(f+g) \leq f+\alpha g$.
\end{lemma}
\begin{lemma}\textup{(\cite[Lemma 2.2]{CHITANGA2026})}\label{l2.7}
    Let $L:C([a,b])\to C([a,b])$ be a positive linear operator. If there exists $u\in C([a,b])$ such that $u$ is strictly positive and $\alpha u\leq Lu\leq \beta u$, then $\alpha\leq r(L)\leq \beta$.
\end{lemma}
In the study of the dimension spectrum, we often require lower and upper bounds for the Hausdorff dimension of sets of continued fraction expansions with digits restricted to a subset of $\N$. The following lemma plays an important role in obtaining these estimates.
\begin{lemma}\textup{(\cite[Lemma 4.1]{CHITANGA2026})}\label{l2.8}
    Let $n\in\N$ and $s\geq0$. If $L_{s,\{n\}}:C\left([0,\frac{1}{n}]\right)\to C\left([0,\frac{1}{n}]\right)$ is defined as $$(L_{s,\{n\}}f)(x)=\left(\frac{1}{n+x}\right)^{2s}f\left(\frac{1}{n+x}\right),$$ then $v_s(x)=\left(\frac{1}{\lambda+x}\right)^{2s}$, where $$\lambda=\frac{n+\sqrt{n^2+4}}{2},$$ is a strictly positive eigenvector of $L_{s,\{n\}}$ with eigenvalue $\lambda^{-2s}$. In particular, $r(L_{s,\{n\}})=\lambda^{-2s}$.
\end{lemma}

\section{The dimension spectrum of $F_{(a_1,a_2)}$ with $a_1\geq2$}
Our goal in this section is to show that the dimension spectrum of continued fraction expansions with coefficients restricted to the generalized Fibonacci sequence starting with an integer $\geq2$ is full. We use the concept of a strict break point to show that an $s\in(0,1)$ is in the dimension spectrum. This idea was previously used in \cite{CHITANGA2026} and is similar to the approach introduced by Kesseb\"ohmer and Zhu in \cite{MR2197868}. For completeness, we include the definition of a strict break point and restate Lemma 3.3 of  \cite{CHITANGA2026}, which will be used in this section.
\begin{definition}\label{def3.1}
Let $F=\{a_1,a_2,\ldots\}\subseteq\N$ with $a_1<a_2<\cdots$. Let $\Lambda$ be a finite subset of $F$, and $0<s<\dim_H(J_F)$. We say that $a_{k_0}> \max \Lambda$ is a strict break point for $(\Lambda,s)$ if 
$\dim_H(J_{\Lambda\cup \{a_{k_0}\}})\geq s$ and $\dim_H(J_{\Lambda\cup \{a_{k_0+1}\}})< s$.
\end{definition}
\begin{lemma}\textup{(\cite[Lemma 3.3]{CHITANGA2026})}\label{l3.2}
    Let $F$ be an infinite subset of $\N$, and $0<s<\dim_H(J_F)$. If for each finite subset $\Lambda$ of $F$ with a strict break point $a_{k_0}\in F$ for $(\Lambda,s)$, we have that $\dim_H(J_{\Lambda\cup T})>s$, where $T=\{a_n\in F:n>k_0\}$, then $s\in DS(F)$.
\end{lemma}
We are now ready to prove Theorem \ref{th1.1}. First, we use Lemma \ref{l3.2} to determine a subset of $[0,\dim_{H}(J_{F_{(a_1,a_2)}})]$ that is always contained in the  dimension spectrum of the generalized Fibonacci sequence $F_{(a_1,a_2)}$.

\begin{theorem}\label{th3.3}
       If $F_{(a_1,a_2)}=\{a_1,a_2,\ldots\}\subset\N$ with $a_1<a_2$ and  $a_n=a_{n-1}+a_{n-2}$ for all $n\geq3$. Then $$\left[0,\min\left\{\frac{1}{2},\dim_{H}(J_{F_{(a_1,a_2)}})\right\}\right]\subset DS(F_{(a_1,a_2)}).$$ 
       \begin{proof}
           Clearly, $0\in DS(F_{(a_1,a_2)})$. Let $0<s<\frac{1}{2}$. Let $\Lambda$ be any finite subset of $F_{(a_1,a_2)}$ and $a_{k_0}\in F_{(a_1,a_2)}$ be a strict break point for $(\Lambda,s)$. So, $r(L_{s,\Lambda\cup\{a_{k_0}\}})\geq1$. If we show that $\dim_{H}(J_{\Lambda\cup T})>s$, where $T=\{a_{k}\in F_{(a_1,a_2)}:k>k_0\}$, then by Lemma \ref{l3.2}, $s\in DS(F_{(a_1,a_2)})$. Let $T_m=\Lambda\cup\{a_{k}\in F_{(a_1,a_2)}:k_0<k\leq k_0+m\}$. Let $v_s$ be the strictly positive eigenvector for $L_{s,\Lambda\cup\{a_{k_0}\}}$. Then,  for $x\in\left[0,\frac{1}{a_1}\right]$, we have 
           \begin{align*}
               L_{s,T_m}v_s(x)&=L_{s,\Lambda}v_s(x)+\suml_{j=1}^{m}\left(\frac{1}{a_{k_0+j}+x}\right)^{2s}v_s\left(\frac{1}{a_{k_0+j}+x}\right)\\
               &\geq L_{s,\Lambda}v_s(x)+\left(\frac{1}{a_{k_0}+x}\right)^{2s}v_s\left(\frac{1}{a_{k_0}+x}\right)\suml_{j=1}^{m}\left(\frac{a_{k_0}}{a_{k_0+j}}\right)^{2s},
           \end{align*}
           as $v_s$ is a decreasing function on $\left[0,\frac{1}{a_1}\right]$, and $\frac{a_{k_0}+x}{a_{k_0+j}+x}\geq \frac{a_{k_0}}{a_{k_0+j}}$. Set $$\alpha_m(s):=\suml_{j=1}^{m}\left(\frac{a_{k_0}}{a_{k_0+j}}\right)^{2s}.$$
           If we show that there exists an integer $m$ sufficiently large such that $\alpha_m(s)>1$, then by Lemma \ref{l2.6} , there exists $\beta>1$ such that 
           \begin{align*}
               L_{s,T_m}v_s(x)&\geq \beta\left(L_{s,\Lambda}v_s(x)+\left(\frac{1}{a_{k_0}+x}\right)^{2s}v_s\left(\frac{1}{a_{k_0}+x}\right)\right)\\
               &=\beta L_{s,\Lambda\cup\{a_{k_0}\}}v_s(x)\\
               &=\beta r(L_{s,\Lambda\cup\{a_{k_0}\}})v_s(x)\\
               &\geq \beta v_s(x).
           \end{align*}
         Then, by Lemma \ref{l2.7}, we get $r(L_{s,T_m})\geq\beta>1$, and therefore, by Theorem \ref{th1.4}, $\dim_{H}(J_{T_m})>s$. As $T_m\subset\Lambda\cup T$, we have that $\dim_{H}(J_{\Lambda\cup T})>s$.\par
           It remains to show that $\alpha_m(s)>1$ for some large value of $m$. Note that $k_0\geq2$; therefore, by Lemma \ref{l2.1}, we get
           \begin{align*}
             \suml_{j=1}^{\infty}\left(\frac{a_{k_0}}{a_{k_0+j}}\right)^{2s}&\geq \suml_{j=1}^{\infty}\left(\frac{1}{2^j}\right)^{2s}\\
             &=\frac{1}{2^{2s}-1}.
           \end{align*}
           Since $0<s<\frac{1}{2}$, $\frac{1}{2^{2s}-1}>1$. Thus, there exists $m\in \N$ such that $\alpha_m(s)>1$. Note that the Dimension spectrum is closed by \cite[Theorem 1.2]{MR3960790}. This completes the proof.
       \end{proof}
    \end{theorem}

As a consequence of Theorem \ref{th3.3}, we show that $F_{(a_1,a_2)}$ has full dimension spectrum in the majority of cases. More precisely, if $a_1\geq3$, or if $a_1=2$ and $a_2\geq6$, then $F_{(a_1,a_2)}$ has full dimension spectrum.
    \begin{theorem}\label{th3.4}
       If $F_{(a_1,a_2)}=\{a_1,a_2,\ldots\}\subset\N$ with $3\leq a_1<a_2$ and  $a_n=a_{n-1}+a_{n-2}$ for all $n\geq3$. Then $$\left[0,\dim_{H}(J_{F_{(a_1,a_2)}})\right]=DS(F_{(a_1,a_2)}).$$ 
       \begin{proof}
        In view of Theorem \ref{th3.3}, it suffices to show that $\dim_{H}(J_{F_{(a_1,a_2)}})<\frac{1}{2}$. Let $G=\{3,4,7,11,\ldots\}$. Since there exists a non-decreasing bijection from $G$ to $F_{(a_1,a_2)}$ for every $(a_1,a_2)$ such that $a_1\geq3$, we have
        $\dim_{H}(J_{F_{(a_1,a_2)}})\leq\dim_{H}(J_G)$ by Proposition \ref{p2.4}. So, it is enough to show that $\dim_{H}(J_G)<\frac{1}{2}$. Let $G^m=\{b_1,b_2,\ldots,b_m\}$, where $b_i$ denotes the $i^{th}$ element of $G$. Let $v_s$ be the positive eigenvector of $L_{s,\{3\}}$ with eigenvalue $\lambda^{-2s}$. Thus, by Lemma \ref{l2.8}, $v_s(x)=\left(\frac{1}{\lambda+x}\right)^{2s}$, where $\lambda=\frac{3+\sqrt{13}}{2}$. Note that $\lambda$ satisfies $\lambda^2-3\lambda-1=0$. For $x\in\left[0,\frac{1}{3}\right]$, we have
        \begin{align*}
           L_{s,G^m}v_s(x)&=\lambda^{-2s}v_{s}(x)+\suml_{k=2}^{m}\left(\frac{1}{b_k+x}\right)^{2s}v_s\left(\frac{1}{b_k+x}\right)\\
           &=\lambda^{-2s}v_{s}(x)+\suml_{k=2}^{m}\left(\frac{1}{\lambda b_k+\lambda x+1}\right)^{2s}\\
            &=\lambda^{-2s}\left(1+\suml_{k=2}^{m}\left(\frac{\lambda+x}{ b_k+ x+\lambda-3}\right)^{2s}\right)v_{s}(x),
        \end{align*}
        as $\frac{1}{\lambda}=\lambda-3$. Observe that $\frac{\lambda+x}{ b_k+ x+\lambda-3}$ is increasing on $\left[0,\frac{1}{3}\right]$. Therefore, we obtain
        \begin{align*}
            \lambda^{-2s}\left(1+\suml_{k=2}^{m}\left(\frac{\lambda+x}{ b_k+ x+\lambda-3}\right)^{2s}\right)&\leq\lambda^{-2s}\left(1+\suml_{k=2}^{m}\left(\frac{\lambda+\frac{1}{3}}{ b_k+\lambda-\frac{8}{3}}\right)^{2s}\right)\\
            &=\lambda^{-2s}+\left(\frac{3\lambda+1}{3\lambda}\right)^{2s}\suml_{k=2}^{m}\left(\frac{1}{ b_k+\lambda-\frac{8}{3}}\right)^{2s}\\
            &=:\mu_{m}(s).
        \end{align*}
        Using the properties of the Fibonacci sequence, it is easy to show that $b_k>k^3$ for all $k\geq17$. Also note that $\lambda^2=3\lambda+1$. Thus, we get
        \begin{align*}
            \mu_{m}(s)&\leq \lambda^{-2s}+\left(\frac{\lambda}{3}\right)^{2s}\suml_{k=2}^{16}\left(\frac{1}{ b_k+\lambda-\frac{8}{3}}\right)^{2s}+\left(\frac{\lambda}{3}\right)^{2s}\intl_{16}^{\infty}\frac{1}{x^{6s}}dx\\
            &=\lambda^{-2s}+\left(\frac{\lambda}{3}\right)^{2s}\suml_{k=2}^{16}\left(\frac{1}{ b_k+\lambda-\frac{8}{3}}\right)^{2s}+\left(\frac{\lambda}{3}\right)^{2s}\frac{16^{-6s+1}}{6s-1}.
        \end{align*}
        By direct computation, we obtain  $\mu_m(s)<1$ for $s=0.49$. Thus, by Lemma \ref{l2.7}, $r(L_{s,G^m})\leq \mu_m(s)<1$. Theorem \ref{th1.4} then implies that $\dim_H(J_{G^m})<0.49$ for all $m$. Thus, by Theorem \ref{th2.3}, it follows that  $\dim_H(J_{G})\leq s$ for $s=0.49$.
       \end{proof}
    \end{theorem}

    \begin{theorem}\label{th3.5}
       If $F_{(a_1,a_2)}=\{a_1,a_2,\ldots\}\subset\N$ with $ a_1=2,\ a_2\geq6$, and  $a_n=a_{n-1}+a_{n-2}$ for all $n\geq3$. Then $$\left[0,\dim_{H}(J_{F_{(a_1,a_2)}})\right]=DS(F_{(a_1,a_2)}).$$ 
       \begin{proof}
       We will show that $\dim_{H}(J_{F_{(a_1,a_2)}})<\frac{1}{2}$. The result then follows from Theorem \ref{th3.3}. Let $\HH=\{2,6,8,14,\ldots\}$. Since there exists a non-decreasing bijection from $\HH$ to $F_{(a_1,a_2)}$ for every $(a_1,a_2)$ such that  $a_1=2$ and $a_2\geq6$, we have
        $\dim_{H}(J_{F_{(a_1,a_2)}})\leq\dim_{H}(J_{\HH})$ by Proposition \ref{p2.4}. So, it is enough to show that $\dim_{H}(J_{\HH})<\frac{1}{2}$. Let $\HH^m=\{b_1,b_2,\ldots,b_m\}$, where $b_i$ denotes the $i^{th}$ element of $\HH$. Let $v_s$ be the positive eigenvector of $L_{s,\{2\}}$ with eigenvalue $\lambda^{-2s}$. Thus, by Lemma \ref{l2.8}, $v_s(x)=\left(\frac{1}{\lambda+x}\right)^{2s}$, where $\lambda=1+\sqrt{2}$. Observe that $\lambda$ satisfies $\lambda^2-2\lambda-1=0$. For $x\in[0,\frac{1}{2}]$, we have
        \begin{align*}
           L_{s,\HH^m}v_s(x)&=\lambda^{-2s}v_{s}(x)+\suml_{k=2}^{m}\left(\frac{1}{b_k+x}\right)^{2s}v_s\left(\frac{1}{b_k+x}\right)\\
           &=\lambda^{-2s}v_{s}(x)+\suml_{k=2}^{m}\left(\frac{1}{\lambda b_k+\lambda x+1}\right)^{2s}\\
            &=\lambda^{-2s}\left(1+\suml_{k=2}^{m}\left(\frac{\lambda+x}{ b_k+ x+\lambda-2}\right)^{2s}\right)v_{s}(x),
        \end{align*}
        as $\frac{1}{\lambda}=\lambda-2$. Since $\frac{\lambda+x}{ b_k+ x+\lambda-2}$ is increasing on $\left[0,\frac{1}{2}\right]$, we get
        \begin{align*}
            \lambda^{-2s}\left(1+\suml_{k=2}^{m}\left(\frac{\lambda+x}{ b_k+ x+\lambda-2}\right)^{2s}\right)&\leq\lambda^{-2s}\left(1+\suml_{k=2}^{m}\left(\frac{\lambda+\frac{1}{2}}{ b_k+\lambda-\frac{3}{2}}\right)^{2s}\right)\\
            &=\lambda^{-2s}+\left(\frac{2\lambda+1}{2\lambda}\right)^{2s}\suml_{k=2}^{m}\left(\frac{1}{ b_k+\lambda-\frac{3}{2}}\right)^{2s}\\
            &=:\mu_{m}(s).
        \end{align*}
       By simple induction, we get $b_k>k^3$ for all $k\geq16$. Also, $\lambda^2=2\lambda+1$. Thus, we get
        \begin{align*}
            \mu_{m}(s)&\leq \lambda^{-2s}+\left(\frac{\lambda}{2}\right)^{2s}\suml_{k=2}^{15}\left(\frac{1}{ b_k+\lambda-\frac{3}{2}}\right)^{2s}+\left(\frac{\lambda}{2}\right)^{2s}\intl_{15}^{\infty}\frac{1}{x^{6s}}dx\\
            &=\lambda^{-2s}+\left(\frac{\lambda}{2}\right)^{2s}\suml_{k=2}^{15}\left(\frac{1}{ b_k+\lambda-\frac{3}{2}}\right)^{2s}+\left(\frac{\lambda}{2}\right)^{2s}\frac{15^{-6s+1}}{6s-1}.
        \end{align*}
       Again, a simple calculation gives $\mu_m(s)<1$ for $s=0.485$. Thus, applying Lemma \ref{l2.7} followed by   Theorem \ref{th1.4}, we obtain $\dim_H(J_{\HH^m})<0.485$ for all $m$. Thus, by Theorem \ref{th2.3}, it follows that  $\dim_H(J_{\HH})\leq s$ for $s=0.485$.
       \end{proof}
    \end{theorem}

    We are now left with only three cases, namely $a_1=2$ and $a_2\in\{3,4,5\}$, in order to complete the proof of Theorem \ref{th1.1}, which we will cover in our next theorem.

\begin{theorem}\label{th3.6}
  If $F_{(a_1,a_2)}=\{a_1,a_2,\ldots\}\subset\N$ with $ a_1=2,\ a_2\in\{3,4,5\}$, and  $a_n=a_{n-1}+a_{n-2}$ for all $n\geq3$. Then $$\left[0,\dim_{H}(J_{F_{(a_1,a_2)}})\right]=DS(F_{(a_1,a_2)}).$$  
  \begin{proof}
      First, we find an upper bound for $\dim_{H}(J_{F_{(a_1,a_2)}})$ in the case where $a_1=2$ and $a_2\in\{3,4,5\}$. Let $\G=\{2,3,5,8,\ldots\}$. Since there exists a non-decreasing bijection from $\G$ to $F_{(a_1,a_2)}$ for every $(a_1,a_2)$ such that  $a_1=2$ and $a_2\in\{3,4,5\}$, we have
        $\dim_{H}(J_{F_{(a_1,a_2)}})\leq\dim_{H}(J_{\G})$ by Proposition \ref{p2.4}. Let $\G^m=\{b_1,b_2,\ldots,b_m\}$, where $b_i$ denotes the $i^{th}$ element of $\G$. Let $v_s$ be the positive eigenvector of $L_{s,\{2\}}$ with eigenvalue $\lambda^{-2s}$. Thus, by Lemma \ref{l2.8}, $v_s(x)=\left(\frac{1}{\lambda+x}\right)^{2s}$, where $\lambda$ satisfies $\lambda^2-2\lambda-1=0$, i.e. $\lambda=1+\sqrt{2}$. Proceeding as in the proof of the previous theorem, we obtain
        \begin{align*}         L_{s,\G^m}v_s(x)&=\lambda^{-2s}\left(1+\suml_{k=2}^{m}\left(\frac{\lambda+x}{ b_k+ x+\lambda-2}\right)^{2s}\right)v_{s}(x)\\  
        &\leq \left(\lambda^{-2s}+\left(\frac{2\lambda+1}{2\lambda}\right)^{2s}\suml_{k=2}^{m}\left(\frac{1}{ b_k+\lambda-\frac{3}{2}}\right)^{2s}\right)v_s(x).\\
        \end{align*}
        Let $\mu_m(s):=\lambda^{-2s}+\left(\frac{2\lambda+1}{2\lambda}\right)^{2s}\suml_{k=2}^{m}\left(\frac{1}{ b_k+\lambda-\frac{3}{2}}\right)^{2s}.$ Observe that $b_k>k^2$ for all $k\geq9$, and using the fact that $\lambda^2=2\lambda+1$, we get
       \begin{align*}
            \mu_{m}(s)&\leq \lambda^{-2s}+\left(\frac{\lambda}{2}\right)^{2s}\suml_{k=2}^{8}\left(\frac{1}{ b_k+\lambda-\frac{3}{2}}\right)^{2s}+\left(\frac{\lambda}{2}\right)^{2s}\intl_{8}^{\infty}\frac{1}{x^{4s}}dx\\
            &=\lambda^{-2s}+\left(\frac{\lambda}{2}\right)^{2s}\suml_{k=2}^{8}\left(\frac{1}{ b_k+\lambda-\frac{3}{2}}\right)^{2s}+\left(\frac{\lambda}{2}\right)^{2s}\frac{8^{-4s+1}}{4s-1}.
        \end{align*}
        Using a calculator we find that $\mu_m(s)<1$ for $s=0.6$. Thus, applying Lemma \ref{l2.7}, and then   Theorem \ref{th1.4}, we obtain that  $\dim_H(J_{\G^m})<0.6$ for all $m$. Thus, by Theorem \ref{th2.3}, it follows that  $\dim_H(J_{\G})\leq s$ for $s=0.6$. Hence, for $a_1=2$ and $a_2\in\{3,4,5\}$, we have 
        $$\dim_{H}(J_{F_{(a_1,a_2)}})\leq 0.6.$$
Now, we will show that $\left[0,\dim_{H}(J_{F_{(a_1,a_2)}})\right]=DS(F_{(a_1,a_2)})$. Clearly $0$ and $\dim_{H}(J_{F_{(a_1,a_2)}})$ are in the dimension spectrum of $F_{(a_1,a_2)}$. Let  $0<s<\dim_{H}(J_{F_{(a_1,a_2)}})$. Then $s\leq0.6$. Let $\Lambda$ be any finite subset of $F_{(a_1,a_2)}$, and $a_{k_0}\in F_{(a_1,a_2)}$ be a strict break point for $(\Lambda,s)$. Following the reasoning in the proof of Theorem \ref{th3.3}, we see that in order to prove $s\in DS(F_{(a_1,a_2)})$, it is enough to show that $\alpha_{m,a_2}(s)>1$ for some sufficiently large $m$ in each  case where $a_1=2$ and $a_2\in\{3,4,5\}$. Here, $$\alpha_{m,a_2}(s):=\suml_{j=1}^{m}\left(\frac{a_{k_0}}{a_{k_0+j}}\right)^{2s}.$$
Clearly, $k_0\geq2$. Therefore, by Lemma \ref{l2.2}, we have
 $$\frac{a_{k_0}}{a_{k_0+j}}\geq \min\left\{\frac{a_2}{a_{2+j}},\frac{a_3}{a_{3+j}}\right\}.$$
 If $a_1=2$ and $a_2=3$, then $D=a_2^2-a_1a_2-a_1^2<0$. Also, $\alpha_{m,a_2}(s)$ is a decreasing function of $s$, and $s\leq0.6$. Thus, we get 
 \begin{align*}
     \suml_{j=1}^{m}\left(\frac{a_{k_0}}{a_{k_0+j}}\right)^{2s}&\geq \suml_{j=1}^{m}\left(\frac{a_2}{a_{2+j}}\right)^{2(0.6)}\\
     &=3^{1.2}\suml_{j=3}^{m+2}\left(\frac{1}{a_{j}}\right)^{1.2}.
 \end{align*}
 Calculating first few terms, we get $\alpha_{m,a_2}(s)>1$ for $m=3$. In the remaining two cases, where $a_1=2$ and $a_2\in\{4,5\}$, we see that $D=a_2^2-a_1a_2-a_1^2>0$. Thus, we have 
 \begin{align*}
     \suml_{j=1}^{m}\left(\frac{a_{k_0}}{a_{k_0+j}}\right)^{2s}&\geq \suml_{j=1}^{m}\left(\frac{a_3}{a_{3+j}}\right)^{2(0.6)}\\
     &=a_3^{1.2}\suml_{j=4}^{m+3}\left(\frac{1}{a_{j}}\right)^{1.2}.
 \end{align*}
       Using a calculator, we find that for $a_2=4$, the inequality  $\alpha_{m,a_2}(s)>1$ holds for $m=3$, and for $a_2=5$, $\alpha_{m,a_2}(s)>1$ for $m=4$. This completes the proof.
  \end{proof}
\end{theorem}
\section{The dimension spectrum of $F_{(a_1,a_2)}$ with $a_1=1$}

In this section, we consider the set of continued fraction expansions with coefficients restricted to the generalized Fibonacci sequence $F_{(a_1,a_2)}$ with $a_1=1$. We shall see that the structure of the dimension spectrum becomes more intricate as $a_2$  decreases. In certain cases, we require explicit lower and upper bounds for $\dim_{H}(J_{A})$ for specific subsets $A\subset\N$. For finite sets $A$, we obtained these bounds using rigorous numerical methods developed by Falk and Nussbaum in \cite{MR3827801,MR3851775} and the MATLAB code from
\begin{gather}
\text{\url{https://sites.math.rutgers.edu/~falk/hausdorff/codes.html}} \label{code}
\end{gather}

For infinite sets, we follow the strategy outlined in \cite{MR4068257}. Below, we provide the bounds for certain infinite subsets of $\N$ that are sufficient for our analysis.\par
We start by recalling the following result from \cite{MR3960790,MR1901102}, which is essential for estimating the Hausdorff dimension of the set of continued fractions with infinite digit sets. Before stating it, we recall that the set of continued fraction expansions can be realized as the limit set of an infinite CIFS $$S_{\Lambda}:=\{\phi_i:[0,1]\to[0,1];x\mapsto1/(i+x):i\in\Lambda\subseteq\N\}.$$ 
 We denote by $K_{\Lambda}$ the distortion constant of $S_{\Lambda}$, and     by $\chi(S_{\Lambda})$ its characteristic Lyapunov exponent. For any subset $F\subseteq\Lambda$, let $h_F:=\dim_{H}(J_{F})$. 
 \begin{theorem}\textup{(\cite{MR3960790,MR1901102})}\label{th4.1}
    Let $S_{\Lambda}=\{\phi_i\}_{i\in\Lambda}$ be a strongly regular CIFS. If $F\subset\Lambda$ is a finite set such that $h_F$ is sufficiently large, then $$\dim_{H}(J_{\Lambda})-\dim_{H}(J_{F})\leq\frac{K_{\Lambda}^{h_F}}{\chi(S_{\Lambda})}\suml_{\Lambda\setminus F}\|\phi_i'\|_{\infty}^{h_F}.$$
 \end{theorem}
For the definitions of $K_{\Lambda}$ and $\chi(S_{\Lambda})$, we refer the reader to \cite{MR1387085,MR2003772}. As we are concerned here only with their bounds,  the following proposition from \cite{MR4068257} provides a lower bound for the Lyapunov exponent of $S_{\Lambda}$.
\begin{proposition}\textup{(\cite[Proposition 2.6]{MR4068257})}\label{p4.2}
    If $\Lambda\subset\{n,n+1,n+2,\ldots\}$ for some $n\in\N$, then 
    $$\chi(S_{\Lambda})\geq2\ln\left(\frac{n+\sqrt{n^2+4}}{2}\right).$$
\end{proposition}
 We also note that if $1\in\Lambda$, then $K_{\Lambda}=4$; see \cite{MR1487636}.

\begin{lemma}\label{l4.3}
   The following upper bounds hold:
   $$\dim_H(J_{F_{(1,3)}})\leq 0.71,\ \dim_{H}(J_{F_{(1,2)}\setminus\{8\}})\leq0.7976.$$
   \begin{proof}
      For some large $M$, define $F_{(1,3)}^M=\{1,3,a_3,\ldots,a_M\}$, where $a_i$ denotes the $i^{th}$ element of  $F_{(1,3)}$. Suppose $l_M$ and $u_M$ are such that $$l_M\leq h_{F_{(1,3)}^M}\leq u_M,$$ where $h_{F_{(1,3)}^M}=\dim_H(J_{{F_{(1,3)}^M}})$. It is easy to see that $a_k>k^3$ for all $k\geq19$. Hence, for sufficiently large $M$, we have $$\suml_{F_{(1,3)}\setminus F_{(1,3)}^M}\|\phi_i'\|_{\infty}^{h_{F_{(1,3)}^M}}\leq \suml_{j\geq M+1}\left(\frac{1}{a_j}\right)^{2l_M}\leq \intl_{M}^{\infty}x^{-6l_M}dx=\frac{M^{-6l_M+1}}{6l_M-1}.$$
      From Proposition \ref{p4.2}, we have $\chi(S_{F_{(1,3)}})\geq 2\ln\left(\frac{1+\sqrt{5}}{2}\right)$, and $K_{F_{(1,3)}}=4$ since $1\in F_{(1,3)}$. Hence, by Theorem \ref{th4.1}, we get
      \begin{align}\label{4.2}
          \dim_{H}(J_{F_{(1,3)}})-\dim_{H}(J_{F_{(1,3)}^M})&\leq\frac{4^{u_M}}{2\ln\left(\frac{1+\sqrt{5}}{2}\right)}\left(\frac{M^{-6l_M+1}}{6l_M-1}\right)\\
          &=:err(M) \notag 
      \end{align}
      It is straightforward to verify that inequality \eqref{4.2} remains valid for the set $F_{(1,2)}\setminus\{8\}$, where $l_M$ and $u_M$ denote the corresponding lower and upper bounds of $\dim_{H}(J_{F_{(1,2)}^M\setminus\{8\}})$. Using the MATLAB code provided in \eqref{code} with $N=200$ and $M=50$, we find that  
      $$0.7096\leq h_{F_{(1,3)}^M}\leq 0.7097,$$ and  $$0.7974\leq h_{F_{(1,2)}^M\setminus\{8\}}\leq 0.7975.$$ Thus, by \eqref{4.2}, we get
      $$\dim_{H}(J_{F_{(1,3)}})\leq0.7097+err(50)\leq0.71,$$ and 
      $$\dim_{H}(J_{F_{(1,2)}\setminus\{8\}})\leq0.7975+err(50)\leq0.7976.$$
   \end{proof}
\end{lemma}

 The following table summarizes the bounds for selected subsets of $\N$, which were obtained by running the MATLAB code \eqref{code} with $N=200$.  Although these bounds are not sharp, they are sufficient for our purpose. 
 \begin{table}[h!]
\centering
\caption{Lower and upper bounds for Hausdorff dimension}
\label{tab1}
\renewcommand{\arraystretch}{1.2} 
\begin{tabular}{|l|l|}
\hline
Subsystem & Hausdorff dim. interval \\ \hline
$\{1, 2\}$ & $[0.531277, 0.531281]$ \\ \hline
$\{1, 2, 3\}$ & $[0.705657, 0.705662]$ \\ \hline
$\{1, 2, 3, 4\}$ & $[0.788941, 0.788947]$ \\ \hline
$\{1, 4, 5, 9\}$ & $[0.596692, 0.596696]$ \\ \hline
$\{1, 3, 4, 7\}$ & $[0.660152, 0.660157]$ \\ \hline
$\{1, 2, 3, 5, 8\}$ & $[0.798644, 0.798650]$ \\ \hline
\end{tabular}
\end{table}

We now prove Theorem \ref{th1.2}. Following the technique introduced in the previous section, we first establish part (i) of the theorem.
\begin{theorem}\label{th4.4}
    If $F_{(a_1,a_2)}=\{a_1,a_2,\ldots\}\subset\N$ with $ a_1=1,\ a_2\geq3$, and  $a_n=a_{n-1}+a_{n-2}$ for all $n\geq3$. Then $$\left[0,\dim_{H}(J_{F_{(a_1,a_2)}})\right]=DS(F_{(a_1,a_2)}).$$
    \begin{proof}
       Let $\F_{a_2}^m=\{1,a_2,\ldots,a_m\}$, where $a_i$ denotes the $i^{th}$ element of $F_{(1,a_2)}$. Let $v_s$ be the positive eigenvector of $L_{s,\{1\}}$ with eigenvalue $\lambda^{-2s}$. Thus, by Lemma \ref{l2.8}, $v_s(x)=\left(\frac{1}{\lambda+x}\right)^{2s}$, where $\lambda=\frac{1+\sqrt{5}}{2}$. Note that $\lambda$ satisfies $\lambda^2-\lambda-1=0$. For $x\in\left[0,1\right]$, we have 
       \begin{align*}         L_{s,\F_{a_2}^m}v_s(x)&=\lambda^{-2s}\left(1+\suml_{k=2}^{m}\left(\frac{\lambda+x}{ a_k+ x+\lambda-1}\right)^{2s}\right)v_{s}(x)\\  
        &\leq \left(\lambda^{-2s}+\left(\frac{\lambda+1}{\lambda}\right)^{2s}\suml_{k=2}^{m}\left(\frac{1}{ a_k+\lambda}\right)^{2s}\right)v_s(x).\\
        \end{align*}
        It is easy to see that $a_k>k^4$ for all $k\geq27$. Thus, we get
        \begin{align*}
            \lambda^{-2s}+\left(\frac{\lambda+1}{\lambda}\right)^{2s}\suml_{k=2}^{m}\left(\frac{1}{ a_k+\lambda}\right)^{2s}&\leq \lambda^{-2s}+\lambda^{2s}\suml_{k=2}^{26}\left(\frac{1}{ a_k+\lambda}\right)^{2s}+\lambda^{2s}\intl_{26}^{\infty}\frac{1}{x^{8s}}dx\\
            &=\lambda^{-2s}+\lambda^{2s}\suml_{k=2}^{26}\left(\frac{1}{ a_k+\lambda}\right)^{2s}+\lambda^{2s}\frac{26^{-8s+1}}{8s-1}\\
            &=:\mu(s,a_2)
        \end{align*}
         By direct computation, we obtain  $\mu(s,a_2)<1$ for $s=0.496$ and $a_2=13$. Thus, for these values of $s$ and $a_2$, Lemma \ref{l2.7} implies that $r(L_{s,\F_{a_2}^m})\leq \mu(s,a_2)<1$. Theorem \ref{th1.4} then implies that $\dim_H(J_{\F_{a_2}^m})<0.496$ for all $m$. Thus, by Theorem \ref{th2.3}, it follows that  $\dim_H(J_{F_{(1,13)}})\leq s$ for $s=0.496$. Since there exists a non-decreasing bijection from $F_{(1,13)}$ to $F_{(1,a_2)}$ for all $a_2\geq13$, we have
        $\dim_{H}(J_{F_{(1,a_2)}})\leq\dim_{H}(J_{F_{(1,13)}})<\frac{1}{2}$ by Proposition \ref{p2.4}. Hence, by Theorem \ref{th3.3}, 
        $$\left[0,\dim_{H}(J_{F_{(1,a_2)}})\right]=DS(F_{(1,a_2)}),$$
        for all $a_2\geq13$.\par
        For $a_2\in\{4,5,6\}$, we find that $\mu(0.7,4)<1,\ \mu(0.65,5)<1$, and $\mu(0.63,6)<1$. Thus, applying Lemma \ref{l2.7}, Theorem \ref{th1.4}, and Theorem \ref{th2.3}, we obtain the following estimates:
        \begin{align}\label{4.3}
          \dim_H(J_{F_{(1,4)}})\leq 0.7, \hspace{0.5cm}   \dim_H(J_{F_{(1,5)}})\leq 0.65, \hspace{0.5cm}   \dim_H(J_{F_{(1,6)}})\leq 0.63. 
        \end{align}
        Then, by Proposition \ref{p2.4}, it follows that $$\dim_H(J_{F_{(1,a_2)}})\leq 0.63$$ for $a_2\in\{6,7,\ldots,12\}$. Now, using these estimates we will show that $\left[0,\dim_{H}(J_{F_{(1,a_2)}})\right]=DS(F_{(1,a_2)})$ for every $a_2\in\{3,\ldots,12\}$. Clearly $0$ and $\dim_{H}(J_{F_{(1,a_2)}})$ are in the dimension spectrum of $F_{(1,a_2)}$. Let  $0<s<\dim_{H}(J_{F_{(1,a_2)}})$. Let $\Lambda$ be any finite subset of $F_{(1,a_2)}$, and $a_{k_0}\in F_{(1,a_2)}$ be a strict break point for $(\Lambda,s)$. As we have already seen that in order to prove $s\in DS(F_{(1,a_2)})$, it is enough to show that $\alpha_{m,a_2}(s)>1$ for some sufficiently large $m$, where $$\alpha_{m,a_2}(s):=\suml_{j=1}^{m}\left(\frac{a_{k_0}}{a_{k_0+j}}\right)^{2s}.$$ 
        Note that for every pair $(1,a_2)$ with $a_2\geq3$, we have $D=a_2^2-a_2-1>0$. Moreover, $k_0\geq2$. Therefore, by Lemma \ref{l2.2}, we have
        $$\suml_{j=1}^{m}\left(\frac{a_{k_0}}{a_{k_0+j}}\right)^{2s}\geq \suml_{j=1}^{m}\left(\frac{a_3}{a_{3+j}}\right)^{2s}.$$ For $a_2\in\{4,5,6\}$, let $h_{a_2}$ denote the corresponding upper bound for $\dim_H(J_{F_{(1,a_2)}})$ given in \eqref{4.3}. Since $\alpha_{m,a_2}(s)$ is a decreasing function of $s$, we have 
        $$\alpha_{m,a_2}(s)\geq a_{3}^{2h_{a_2}}\suml_{j=4}^{m+3}\left(\frac{1}{a_{j}}\right)^{2h_{a_2}}.$$
        For $a_2\in\{6,7,\ldots,12\}$, we have $h_{a_2}=0.63$, and computation shows that $\alpha_{m,a_2}(s)>1$ for $m=15$.\\ For $a_2=5$, $h_{a_2}=0.65$. Using a calculator, we find that the inequality $\alpha_{m,a_2}(s)>1$ holds for $m=10$. The remaining two cases, namely $a_2\in\{3,4\}$, require slightly more effort. We first consider the case $a_2=4$. By \eqref{4.3}, we have $$\dim_H(J_{F_{(1,4)}})\leq0.7.$$ 
        In this case, we find that $$\alpha_{m,a_2}(0.66)\geq a_{3}^{2(0.66)}\suml_{j=4}^{m+3}\left(\frac{1}{a_{j}}\right)^{2(0.66)}>1$$ for $m=10$. Since $\alpha_{m,a_2}(s)$ is decreasing in $s$, it follows that for $m=10$, $\alpha_{m,a_2}(s)>1$ for all $0<s\leq0.66$. Consequently, we obtain $$[0,0.66]\subset DS(F_{(1,4)}).$$ Now let $0.66<s<\dim_H(J_{F_{(1,4)}})$. By definition \ref{def3.1} of the strict break point, we have $\dim_H(J_{\Lambda\cup \{a_{k_0}\}})\geq s>0.66$. On the other hand, Table \ref{tab1} shows that $\dim_H(J_{ \{1,4,5,9\}})\leq 0.597$, which implies that $k_0\geq5$. It follows that $$\alpha_{m,a_2}(s)\geq a_{5}^{2h_{a_2}}\suml_{j=6}^{m+5}\left(\frac{1}{a_{j}}\right)^{2h_{a_2}}.$$
        Using $h_{a_2}=0.7$ and calculating the first few terms, we find that  $\alpha_{m,a_2}(s)>1$ for $m=6$. This completes the case $a_2=4$.\par
        Finally, we consider the case $a_2=3$. The argument is similar to the previous case, however a sharper upper bound for $\dim_H(J_{F_{(1,3)}})$ is required. By Lemma \ref{l4.3}, we have $\dim_H(J_{F_{(1,3)}})\leq 0.71$. Then, it can be verified that, for $m=7$, $$\alpha_{m,a_2}(s)>1\ \text{for all}\ s\leq0.67.$$ Moreover, for $0.67<s<\dim_H(J_{F_{(1,3)}})$, the definition of $a_{k_0}$, together with the bound $\dim_H(J_{ \{1,3,4,7\}})\leq 0.661$ (see Table \ref{tab1}), implies that $k_0\geq5$. Thus, we have  
        $$\alpha_{m,a_2}(s)\geq a_{5}^{2h_{a_2}}\suml_{j=6}^{m+5}\left(\frac{1}{a_{j}}\right)^{2h_{a_2}}>1$$ for $h_{a_2}=0.71$ and $m=8$. This completes the case $a_2=3$, and hence the proof of the theorem.
    \end{proof}
\end{theorem}

We will require the following lemma to find the gaps in the dimension spectrum. The idea is similar to that used in \cite[Theorem 1.2]{CHITANGA2026}.
\begin{lemma}\label{l4.5}
   Let $F=\{a_1,a_2,\ldots\}\subseteq\N$ with $a_1<a_2<\cdots$. If there exists $a_k\in F$ such that $\dim_{H}(J_{F\setminus\{a_k\}})<\dim_{H}(J_{\{a_1,a_2,\ldots,a_k\}})$, then 
$$\left(\dim_{H}(J_{F\setminus\{a_k\}}),\dim_{H}(J_{\{a_1,a_2,\ldots,a_k\}})\right)\cap DS(F)=\varnothing.$$
\begin{proof}
   Assume that there exists a subset $E\subset F$ such that $$\dim_{H}(J_{F\setminus\{a_k\}})<\dim_{H}(J_{E})<\dim_{H}(J_{\{a_1,a_2,\ldots,a_k\}}).$$ Clearly, $a_k\in E$. Otherwise, we would have  $E\subset F\setminus\{a_k\}$, and hence $\dim_{H}(J_{E})\leq\dim_{H}(J_{F\setminus\{a_k\}})$, which contradicts our assumption. We claim that $\{a_1,a_2,\ldots,a_{k-1}\}\subset E$. Suppose, to the contrary, that $i\in\{1,2,\ldots,k-1\}$ is the smallest integer such that $a_i\notin E$. Define $E_i=E\setminus\{a_k\}\cup\{a_i\}$. Then it is easy to see that there exists a non-decreasing bijection from $E_i$ to $E$. By Proposition \ref{p2.4}, it follows that $\dim_{H}(J_{E})\leq\dim_{H}(J_{E_i})$. Since $E_i\subset F\setminus\{a_k\}$, we have $$\dim_{H}(J_{E})\leq\dim_{H}(J_{E_i})\leq\dim_{H}(J_{F\setminus\{a_k\}}),$$ which is a contradiction. Thus, $\{a_1,a_2,\ldots,a_{k}\}\subset E$, and therefore $$\dim_{H}(J_{E})\geq\dim_{H}(J_{\{a_1,a_2,\ldots,a_k\}}),$$ Which is again a contradiction. Hence, no such set $E$ exists, and consequently $$\left(\dim_{H}(J_{F\setminus\{a_k\}}),\dim_{H}(J_{\{a_1,a_2,\ldots,a_k\}})\right)\cap DS(F)=\varnothing.$$
\end{proof}
\end{lemma}
So far, we have shown that $F_{(a_1,a_2)}$ has full dimension spectrum in all cases except $(a_1,a_2)=(1,2)$. The next theorem addresses this remaining case and completes the proof of Theorem \ref{th1.2}.

    \begin{theorem}\label{th4.6}
        Let $F_{(a_1,a_2)}=\{a_1,a_2,\ldots\}\subset\N$ with $a_n=a_{n-1}+a_{n-2}$ for all $n\geq3$. Then 
        \begin{enumerate}
            \item[(i)] there exists $0<s<\dim_{H}(J_{F_{(1,2)}})$ such that $[0,s]\subset DS(F_{(1,2)})$,
        \item[(ii)] there exist  $0<a<b<\dim_{H}(J_{F_{(1,2)}})$ such that $(a,b)\cap DS(F_{(1,2)})=\varnothing$.
        \end{enumerate}
        \begin{proof}
            It follows from Theorem \ref{th3.3} that $$\left[0,\min\left\{\frac{1}{2},\dim_{H}(J_{F_{(a_1,a_2)}})\right\}\right]\subset DS(F_{(a_1,a_2)}),$$ which proves statement (i). Moreover, using Lemma \ref{l2.2} together with arguments similar to those in the previous proofs, one can further show that $\left[0,0.72\right]\subset DS(F_{(1,2)})$.\par
            To prove statement (ii), it suffices to find $a_k\in F_{(1,2)}$ such that $$\dim_{H}(J_{F_{(1,2)}\setminus\{a_k\}})<\dim_{H}(J_{\{a_1,a_2,\ldots,a_k\}}).$$ The result then follows from Lemma \ref{l4.5}. We use the explicit lower and upper bounds given in Table \ref{tab1}. In particular, $\dim_{H}(J_{\{1,2,3\}})\leq0.71$, while $\left[0,0.72\right]\subset DS(F_{(1,2)})$. Hence, any such $k$ must satisfy $k\geq4$. We now verify that this occurs for $a_k=8$. By Lemma \ref{l4.3}, $\dim_{H}(J_{F_{(1,2)}\setminus\{8\}})\leq0.7976$, whereas Table \ref{tab1} gives $\dim_{H}(J_{\{1,2,\ldots,8\}})\geq0.7986$. Thus,
            $$\dim_{H}(J_{F_{(1,2)}\setminus\{8\}})<\dim_{H}(J_{\{1,2,\ldots,8\}}),$$
            which completes the proof.
        \end{proof}
    \end{theorem}
\subsection{The dimension spectrum of Lucas numbers}
We conclude the paper with a brief comment on the proof of Theorem \ref{th1.3}. Since the argument proceeds along the same lines as the proof of Theorem \ref{th4.6}, we omit the details.

\medskip
\noindent\textbf{Proof of Theorem \ref{th1.3}.} Statement (i) follows directly from Theorem \ref{th3.3}. Furthermore, using Lemma \ref{l2.2}, it can be shown that $\left[0,0.67\right]\subset DS(F_{(2,1)})$. Analogously to Lemma \ref{l4.3}, we obtain $$\dim_{H}(J_{F_{(2,1)}\setminus\{4\}})\leq0.7867,$$
whereas Table \ref{tab1} shows that
$\dim_{H}(J_{\{1,2,3,4\}})\geq0.7889.$ Consequently,
            $$\dim_{H}(J_{F_{(2,1)}\setminus\{4\}})<\dim_{H}(J_{\{1,2,3,4\}}),$$
            which completes the proof.\\

	\noindent {\bf Statements and Declarations:}
\begin{itemize}
	\item {\bf Competing Interests:} The authors have no relevant financial or non-financial interests to disclose.
	\item {\bf Data availability:} Data sharing is not applicable to this article as no datasets were generated or analysed during the current study.
	\item 	{\bf Author Contributions:} All authors contributed equally in this manuscript.
\end{itemize}

\section*{Acknowledgement}
The first author thanks the MHRD, India, for financial support in the form of a Senior
Research Fellowship at the Indian Institute of Technology Delhi.

        \nocite{*}
\bibliographystyle{plain}

\bibliography{ref}
\end{document}